\documentclass{amsart}[18pt]
\usepackage[T1]{fontenc}
\usepackage[left=2cm, right=2cm, top=2cm]{geometry}
\usepackage{amssymb}

\usepackage[cmtip,all]{xy}

\usepackage{hyperref}

\newtheorem{theorem}{Theorem}[section]
\newtheorem{lemma}[theorem]{Lemma}
\newtheorem{setup}[theorem]{Set-up}

\theoremstyle{definition}
\newtheorem{definition}[theorem]{Definition}
\newtheorem{example}[theorem]{Example}

\theoremstyle{remark}
\newtheorem{remark}[theorem]{Remark}
\newtheorem{question}[theorem]{Question}

\numberwithin{equation}{section}

\renewcommand{\b}[1]{\mathbf{#1}}
\newcommand{\sX}{\mathcal{X}}
\newcommand{\sY}{\mathcal{Y}}
\newcommand{\sA}{\mathcal{A}}
\newcommand{\bL}{\mathbb{L}}

\begin{document}

\title[Formal deformations of algebraic spaces...]{Formal deformations of algebraic spaces and generalizations of the motivic Igusa-zeta function.}


\author{Andrew R. Stout}
\address{Borough of Manhattan Community College, CUNY}
\curraddr{199 Chambers Street\\ New York, NY 10007}
\email{astout@bmcc.cuny.edu}
\thanks{Support for this project was provided by a PSC-CUNY Award (PSC-Grant Traditional B, \# 60784-00 48), jointly funded by the Professional Staff Congress and The City University of New York.}



\date{}

\begin{abstract}
We generalize the notion of the auto-Igusa zeta function to formal deformations of algebraic spaces. By incorporating data from all algebraic transformations of local coordinates, this function can be viewed as a generalization of the traditional motivic Igusa zeta function. Furthermore, we introduce a new series, which we term the canonical auto-Igusa zeta function, whose coefficients are given by the quotient stacks formed from the coefficients of the auto-Igusa zeta function modulo change of coordinates. We indicate the current state of the literature on these generalized Igusa-zeta functions and offer directions for future research.
\end{abstract}

\maketitle


\section{Introduction} 

Progress has been made recently in understanding certain motivic generating series associated to the germ of a variety. Originally introduced by Schoutens in \cite{sch1} and \cite{sch2}, the auto-Igusa zeta series is a motivic generating series whose coefficients are determined by the algebra automorphisms of the local ring corresponding to a point on a variety. First calculations of this series were completed by Schoutens and secondary calculations were carried out by the author of this paper with the use of computer algebra software in \cite{STOUT2017156}. This led to several conjectures concerning the structure of the auto-Igusa zeta function. The author later confirmed that these conjectures are indeed true in \cite{stout2} for plane curve singularities, yet the veracity of these conjectures for broader classes of singularities remains unknown. 

In this paper, the author generalizes the auto-Igusa zeta function to stacks.  However, we choose to frame almost all of the work in the language of algebraic spaces, as there is a bottleneck for Artin stacks related to the occasional lack of representability of the hom functor for general Artin stacks. The representability results are discussed in \S \ref{rep}. 

One notable development in this paper concerns the introduction of a new type of moduli stack, which takes into account the canonical group action on the space of endormophism by conjugation of all automorphisms. We call this the moduli stack of Jordan-Normal forms.  This space is in many ways more natural than the space of all endomrophism (from which the auto-Igusa zeta series is formed), yet it is unclear to the author  under what conditions this moduli stack is itself an algebraic space or even when it is an Artin stack. Regardless, one may form a generating series from this stack, which we then term the canonical auto-Igusa zeta function. 

Another important development in this paper is that it is shown in \S\ref{gen} that both the canonical  zeta function and the auto-Igusa zeta function may both be regarded as generalizations of the motivic Igusa zeta function, which has been studied by Denef, Loeser, Cluckers, Nicaise, Mustata, and many others. It is hoped that the motivic Igusa zeta series would offer a key to proving the monodromy conjecture first introduced by Igusa. However, as of yet, the general conjecture remains unproven. 

In \S\ref{last}, the author rephrases his work in \cite{stout2} in this new context and offers possible directions for future research. As mentioned at the beginning of this section, the possible rationality of the auto-Igusa zeta function for varieties (and other conjectures) other than plane curves remains mostly elusive to direct proof. It is the hope of the author that the work carried out in this paper will offer progress toward such a proof as well as possible keys to further understanding of local zeta functions in general.

\section{Representability Results Concerning Hom} \label{rep}

Let $\b{C}$ be a site. It follows from the Yoneda lemma that the Grothendieck topos $\tau = \mbox{Sh}(\b{C})$ is a locally cartesian closed monoidal category. In particular, given two objects $\sX$ and $\sY$ of $\tau$, we have a presheaf which sends an object $U$ of $\b{C}$ to the set $\mbox{Hom}_{\tau}(\sX\otimes U , \sY\otimes U)$, and, moreover, this presheaf is in fact a sheaf -- i.e., it is an object of $\tau$, which we denote by $\underline{\mbox{Hom}}_{\tau}(\sX,\sY)$.  In other words, there is a natural bijection 
\begin{equation}
\mbox{Hom}_{\tau}(\sX\otimes U , \sY\otimes U) \xrightarrow{\sim} \mbox{Hom}_{\tau}(U,\underline{\mbox{Hom}}_{\tau}(\sX,\sY)).
\end{equation}

In fact, a similar result holds in higher category theory--i.e., if $\tau$ is an $(\infty,1)$-topos,  
then $\tau$ is a locally cartesian closed monoidal $(\infty,1)$-category. This follows from the $(\infty,1)$-Yoneda lemma (cf. Proposition 5.1.3.1 of \cite{HTT}). In particular, if $\sX$ and $
\sY$ are two stacks over an algebraic space $S$, then one defines $\underline{\mbox{Hom}}_{S}(\sX, \sY)
$ to be the fibered category of groupoids over the site of $S$-schemes $(\b{Sch}/S)_J$ with a 
given Grothendieck topology $J$, which to any obect $U\to S$ of  $(\b{Sch}/S)_J$ associates the 
groupoid of functors from $\sX\times_S U$ to  $\sY\times_S U$ over $U$, and, moreover, this fibered category 
$\underline{\mbox{Hom}}_{S}(\sX, \sY)$ is in fact a stack. It is possible to show that $\underline{\mbox{Hom}}_{S}(\sX, \sY)$ is an Artin stack (or, Deligne-Mumford stack, or algebraic space) depending on certain conditions on $\sX$ and $\sY$ (cf. Theorem 1.1 of \cite{olsson2006}). For example, we have the following result due to Artin. 

\begin{theorem} If $S$ is a locally Noetherian algebraic space and if $X$ and $Y$ are algebraic spaces over $S$ with $X\to S$ proper and flat and $Y\to S$ is separated and of finite type, then $\underline{\mbox{Hom}}_{S}(X, Y)$ is a separated algebraic space locally of finite type over $S$.
\end{theorem}

This result can be found in \cite{ArtinI}, and generalizes\footnote{Essentially, Artin relaxed the condition that $X\to S$ is projective to merely proper, but this came at the cost of moving from the category of schemes to the larger category of algebraic spaces.} the following result of Grothendieck. 
 If $S$ is a locally Noetherian scheme and if $X$ and $Y$ are in $\b{Sch}/S$ with $X$ projective and flat over $S$ and $Y$ quasi-projective over $S$, then the functor from $\b{Sch}/S$ to $\b{Sets}$ defined by sending $U$ to $\mbox{Hom}_{S}(X\times_S U, Y\times_S U)$ is represented by a scheme $\underline{\mbox{Hom}}_{S}(X, Y)$ which is a separated and locally of finite type over $S$.  This result can be found at the end of page 221-19 of \cite{FGA}. 
 
In particular, if $X$ is finite and flat over a locally Noetherian scheme $S$, then  for any quasi-projective object $Y$ in $\b{Sch}/S$, $\underline{\mbox{Hom}}_{S}(X, Y)$ is a scheme over $S$. This is found in Proposition 5.7 on pages 221-27  of \cite{FGA}, and it usually goes by the name Weil restriction or restriction of scalars.  If $Y \to S$ is separated (respectively, affine, finite type), then $\underline{\mbox{Hom}}_{S}(X, Y)\to S$ is separated (respectively, affine, finite type) over $S$ (cf. Section 7 on pages 221-27 and 221-28 of \cite{FGA}).  

\begin{remark}\label{rem1}
In fact, regarding either Artin's result or Grothendieck's result, the assumption that $S$ is locally Noetherian can be relaxed if we assume that $X \to S$ and $Y\to S$ are locally of finite presentation. Note that a finite type morphism with a locally Noetherian target is locally of finite presentation.
\end{remark}

\section{Projective Systems of Hilbert Spaces}

Let $X \to S$ be a proper and flat morphism of algebraic spaces with $S$ locally Noetherian. There is  a functor $\mbox{Hilb}_{X/S} : \b{Sch}/S \to \b{Sets}$ defined by 
\begin{equation}
\mbox{Hilb}_{X/S}(T) = \{[j: Z \hookrightarrow X\times_S T] \mid Z\to T \mbox{ proper and flat, and } j \mbox{ a closed immersion}\},
\end{equation}
where $[j: Z \hookrightarrow X\times_S T$ denotes the isomorphism class of $j$. 
 This is in fact represented by  a separated algebraic space locally of finite type over $S$ (cf. Section 6 of \cite{ArtinI}) which we denote by $\underline{\mbox{Hilb}}_{X/S}$. The natural transformation
$\iota : \mbox{Hom}_S(X,Y) \hookrightarrow \mbox{Hilb}_{(X\times_SY)/S}$  defined by 
\begin{equation}
\begin{split}
\iota(T) : \mbox{Hom}_S(X,Y)(T) &\hookrightarrow \mbox{Hilb}_{(X\times_SY)/S}(T) \\
[f: X\times_S T \to Y\times_S T] &\mapsto [\Gamma_f \hookrightarrow X\times_S \times Y \times_S T]
\end{split}
\end{equation}
where $\Gamma_f$ denotes the graph of $f$ induces an open immersion of algebraic spaces  $\underline{\mbox{Hom}}_S(X,Y) \hookrightarrow \underline{\mbox{Hilb}}_{(X\times_SY)/S}$ provided $X \to S$ is proper and flat and $Y \to S$ is separated and of finite type. 

Now, if $f: S^{\prime\prime} \rightarrow S^{\prime}$ is an $S$-morphism of algebraic spaces, then there is a natural pullback $f^*$ given by the natural transformation $f^{*}: \mbox{Hilb}_{S^{\prime}/S} \to \mbox{Hilb}_{S^{\prime\prime}/S}$ defined by
\begin{equation}
\begin{split}
f^*(T) : \mbox{Hilb}_{S^{\prime}/S}(T) &\to \mbox{Hilb}_{S^{\prime\prime}/S}(T) \\
Z &\mapsto Z\times_{S^{\prime}\times_S T}(S^{\prime\prime}\times_S T).
\end{split}
\end{equation}

Given two $S$-morphisms $f : X^{\prime} \to X$ and $g: Y^{\prime} \to Y$ of algebraic spaces with $X^{\prime}\to S$ and $X\to S$ proper and flat and with $Y^{\prime}\to S$ and $Y\to S$ separated and finite type, we have the canonical morphism $(f,g) : X^{\prime}\times_SY^{\prime} \to X\times_SY$ and the corresponding  pullback morphism of Hilbert spaces
\begin{equation}
(f,g)^* : \underline{\mbox{Hilb}}_{X\times_SY/S} \to \underline{\mbox{Hilb}}_{X^{\prime}\times_SY^{\prime}/S}.
\end{equation}
If, in addition, $g:Y^{\prime} \hookrightarrow Y$ is a closed immersion, then the restriction of $(f,g)^*$ to $\underline{\mbox{Hom}}(X,Y)$ induces a well-defined morphism of algebraic spaces 
\begin{equation}
(f,g)^* : \underline{\mbox{Hom}}_S(X,Y) \to \underline{\mbox{Hom}}_S(X^{\prime},Y^{\prime}).
\end{equation}
Note that $\underline{\mbox{Hom}}_S(X,Y)$ is contravariant in the first argument and covariant in the second; however, the morphism $(f,g)^*$ sends a morphism $h: X \to Y$ to $h\circ f : X^{\prime}\to Y$ and then restricts the range of $h\circ f$ to the closed subset $Y^{\prime}$ of $Y$.  

An injective system $\{P_i\}_{i\in I}$ of algebraic spaces over $S$ with affine transition maps and with $P_i\to S$ proper and flat gives rise to a projective system $\{\underline{\mbox{Hilb}}_{P_i/S}\}_{i\in I}$ of algebraic spaces with affine transition maps, and it is therefore the case that $\varprojlim\underline{\mbox{Hilb}}_{P_i/S}$ exists in the category of algebraic spaces over $S$.   In particular, if $\{Y_i\}_{i\in I}$ is an injective system whose transition maps are closed immersions with $Y_i \to S$ separated and finite type and if $\{X_i\}_{i\in I}$ is any other injective system with $X_i\to S$ proper and flat, then $\varprojlim\underline{\mbox{Hom}}_S(X_i,Y_i) \hookrightarrow \varprojlim\underline{\mbox{Hilb}}_{X_i\times_S Y_i/S}$ is an open immersion of algebraic spaces over $S$.

Furthermore, if $\{X_i\}_{i\in I}$ is an injective system of algebraic spaces over $S$ whose transitions maps are closed immersions and where each $X_i \to S$ is proper and flat then there is a projective system of open immersions 
\begin{equation}
\underline{\mbox{Aut}}_S(X_i) \hookrightarrow \underline{\mbox{End}}_S(X_i) 
\end{equation}
where $\underline{\mbox{Aut}}(X_i)$ is the algebraic space which represents the subfunctor of $\mbox{Hom}_S(X_i,X_i)$ which associates to $T\to S$ the set of all automorphisms of $X_i\times_ST$. Here, we write  $\underline{\mbox{End}}_S(X_i)$ in place of $\underline{\mbox{Hom}}_S(X_i, X_i).$  Note that $\underline{\mbox{End}}_S(X_i)$ is a monoid object and $\underline{\mbox{Aut}}_S(X_i)$ is a group object in the category of algebraic spaces over $S$. Since the transition maps are closed immersions, the inclusion
\begin{equation}
\varprojlim\underline{\mbox{Aut}}_S(X_i) \hookrightarrow\varprojlim\underline{\mbox{End}}_S(X_i) 
\end{equation}
is an open immersion of algebraic spaces over $S$.

\begin{remark}
As in Remark \ref{rem1}, we may relax the condition that $S$ is locally Noetherian in this section by assuming that all morphisms are locally of finite presentation.
\end{remark}

\section{Finite Free Algebras}\label{3}

\begin{lemma}\label{nak}
Let $X \to S$ be a finite and flat morphism of algebraic spaces with $S$ locally Noetherian. Then, $X\to S$ is finite locally free. In particular, $X\to S$ is faithfully flat. 
\end{lemma}

 \begin{proof}
 Note that flat is local on the source and target in the \'{e}tale topology (cf. Remark 5.4.13 of \cite{stacks}).
Therefore, it is enough to check when $S$ is affine by Lemma 45.3 of \cite{stacks-project}. Since $X\to S$ is finite and hence also an affine morphism, $X$ is affine. Hence, we reduce to the case where $(A,\mathfrak{m})$ is a Noetherian local ring with residue field $k = A/\mathfrak{m}$ and a flat morphism $A \to B$, which induces a surjective morphism $A^n \to B$ of $A$-modules (by finiteness). Therefore, we have a short exact sequence
 $$0 \to I \to A^n \to B \to 0$$ 
 where $I$ is finitely generated (since $A$ is Noetherian). Tensoring this short exact sequence by $k$ gives the short exact sequence 
 $$ 0 \to I\otimes_Ak \to k^n \to B/\mathfrak{m}B \to 0$$
since $B$ is flat over $A$. Since $B/\mathfrak{m}B$ is a free $k$-module, we may choose $n$ above (i.e., generators $b_1,\ldots,b_n$ of $B$ over $A$) so that $B/\mathfrak{m}B$ is isomorphic to $k^n$, which implies that $I/\mathfrak{m}I= I\otimes_Ak = 0$. Therefore, by Nakayama's Lemma (cf, Part (2) of Lemma 10.19.1 of \cite{stacks-project}), $I=0$, and $B$ is a free $A$-module. 
 \end{proof}
 
 \begin{remark}
 The Noetherian condition on $S$ in Lemma \ref{nak} may be relaxed. In fact,  the morphism $X\to S$ of algebraic spaces is a finite, flat, and locally of finite presentation  if and only if it is a finite locally free morphism. This follows from 1.4.7 of \cite{EGA4}, Lemma 45.3 \cite{stacks-project}, and the fact that flatness and locally of finite presentation are both local on the source and target in the \'{e}tale topology (cf. Remark 5.4.13 of \cite{stacks}).
 \end{remark}
 
 Assume that $A$ is a Noetherian local ring with $B$ finite and flat over $A$, and whence free. Let $A^n \xrightarrow{\sim} B$ be a presentation of $B$ as a free $A$-algebra obtained by sending the standard basis elements $e_i$ of $A^n$ to the generating elements $b_i$ of $B$ for $i=1,\ldots,n$. Consider the polynomial ring $R=A[x_1, \ldots, x_n]$ and the surjective homomorphism $\phi : R \to B$ defined by $\phi(x_i) = b_i$ for $i=1,\ldots, n $. Then, by Hilbert's Nullstellensatz, $R$ is Noetherian and therefore the kernel $\mbox{Ker}(\phi)$ is finitely generated -- i.e., $\mbox{Ker}(\phi) = (f_1,\ldots,f_s)$. Moreover, there is a canonical isomorphism $R/\mbox{Ker}(\phi) \xrightarrow{\sim} B$ in the category of $A$-algebras. Therefore, an $A$-algebra endomorphism $\gamma$ of $B$ is given by a choice $\gamma(x_i) = P_i$ with $P_i$ any element of $B=A[x_1,\ldots,x_n]/(f_1,\ldots,f_s)$ for each $i=1,\ldots,n$ such that $f_j(P_1,\ldots,P_n) = 0$ in $B$ for $j=1,\ldots, s$.
 
 \begin{lemma}\label{lem}
 Let $S = \mbox{Spec}(A)$ with $A$ is a reduced Noetherian local ring and let $X = \mbox{Spec}(B)$ with $B=A[x_1,\ldots,x_d]/(x_1,\ldots,x_d)^n$. Set $r=d\cdot(\ell-1)$, where $\ell$ is the (generic) rank of $B$ over $A$.  Then, $X \to S$ is finite and flat, and we have the following isomorphisms
 \begin{equation}
 (\underline{\mbox{End}}_{S}(X))^{\mathrm{red}} \cong \mathbb{A}_{S}^{r}  \quad \mbox{and} \quad
(\underline{\mbox{Aut}}_S(X))^{\mathrm{red}} \cong \mbox{GL}_{d,S} \times_S \mathbb{A}_{S}^{r-d^2}.
\end{equation}
\end{lemma}

\begin{proof}
Consider the $A$-algebra endomorphism of
$A[x_1,\ldots,x_d]/(x_1,\ldots,x_d)^n$ defined by $x_i\mapsto P_i$, where 
we define $P_i=\sum_{|j|<n}a_{i,j}x^j$ where $j$ is a multi-index (i.e., 
$j=(j_1,\ldots,j_d)$, $x^j = \prod_{s=1}^{d}x_s^{j_s}$, and 
$|j|=\sum_{s=1}^{d}j_s< n$) and where $a_{i,j}\in A$ for $i=1,\ldots,d$.
Then, the equations defining the space of $S$-endomorphisms of $X$ are given by 
\begin{equation*}
0=P_{i}^{n} = (a_{i,0} + Q_i)^n, \quad \forall i =1,\ldots,d
\end{equation*}
where $0$ is treated as the multi-index $(0,\ldots,0)$ and 
$Q_i\in(x_1,\ldots,x_d)/(x_1,\ldots,x_d)^n$. This implies 
that $0=(a_{i,0})^n$ for all $i=1,\ldots,d$. Thus, in the reduction, 
$0=a_{i,0}$ for all $i=1,\ldots,d$ since $A$ is reduced. Clearly, $Q^n=0$ 
for all $Q\in (x_1,\ldots,x_d)/(x_1,\ldots,x_d)^n$ from which the first isomorphism follows.
 
 Note that those endomorphisms above which also lie in $(\underline{\mbox{Aut}}_S(X))^{\mathrm{red}}$ are given by $x_i \mapsto \sum_{j=1}^{d}a_{i,j}x_j + T_i$ with $T_i \in (x_1,\ldots,x_d)^2/(x_1,\ldots,x_d)^n$ and where the $d\times d$-matrix $M=(a_{i,j})$ is invertible over $A$ -- i.e., $\mbox{det}(M)$ is a unit of $A$. The second isomorphism then follows from this fact. 
\end{proof}

\begin{remark}
 In the case where $d=1$, we obtain
\begin{equation}
 (\underline{\mbox{End}}_{S}(X))^{\mathrm{red}}  \cong \mathbb{A}_{S}^{n-1} \quad \mbox{and} \quad
(\underline{\mbox{Aut}}_S(X))^{\mathrm{red}}  \cong \mathbb{G}_{m,S} \times_S \mathbb{A}_{S}^{n-2}.
\end{equation}
\end{remark} 

One simple case where Lemma \ref{nak} may be applied is when $S=\mbox{Spec}(A)$ with $(A,\mathfrak{m})$ an Artinian local ring. This is because if $B$ is a finitely generated $A$-module with $A$ an Artinian local ring, then $B$ will have finite length over the residue field $k=A/\mathfrak{m}$ and hence $B$ will also be an Artinian local ring.  In particular, $B$ is automatically free over $A$. Thus, we may form the full subcategory $\b{Fat}/A$ of $\b{Sch}/A$ defined by all finite maps $X \to \mbox{Spec}(A)$. We call\footnote{In general, if $S$ is an algebraic space, we call the full subcategory $\b{C}$ of $\b{Sch}/S$ whose objects are finite, flat, and locally finitely presented over $S$ {\it the category of fat points over} $S$, and we denote this category by $\b{Fat}/S$.} this category the {\it category of fat points over} $A$. Usually, we assume $A=k$ is a field (and, in this case, Lemma \ref{lem} will apply), and often, we will further assume that it has characteristic zero and that it is algebraically closed.

In general, if $X$ is proper and flat over an algebraic space $S$, then the group action
\begin{equation}
\underline{\mbox{End}}_S(X)^{\mathrm{red}}\times_S\underline{\mbox{Aut}}_S(X)^{\mathrm{red}} \to \underline{\mbox{End}}_S(X)^{\mathrm{red}}
\end{equation}
given by conjugation is well-defined in the category of algebraic spaces over $S$. The resulting quotient stack 
\begin{equation}
\mathcal{M}_{\mathrm{JN}}(X):=[\underline{\mbox{End}}_S(X)^{\mathrm{red}}/\underline{\mbox{Aut}}_S(X)^{\mathrm{red}}]
\end{equation}
is termed {\it the moduli stack of Jordan norm forms of } $X$ {\it over} $S$, and it is of general interest. In the case of Lemma \ref{lem}, assuming further that $A$ is an algebraically closed field $k$, the points of $\mathcal{M}_{\mathrm{JN}}(X)$ correspond to $d \times d$ Jordan normal forms over $k$.

\section{Auto-Arc Spaces of Formal Deformations}
\begin{definition}
Let $S$ be an algebraic space and let $I$ be a directed countable set. Let $\{X_{i}\}_{i\in I}$ be an injective system of objects of $\b{Fat}/S$ such that all transition maps $X_i \hookrightarrow X_j$ are closed immersions. We say such a system is {\it an admissible system of fat points over} $S$ if 

1) the structure morphism $j: X_{0} \to S$ is surjective \'{e}tale  and

2) the underlying topological spaces of $X_0$ and $X_i$ are homeomorphic $(\forall i \in I)$. 
\end{definition}

\begin{definition}
 We say that a system $f_i:Y_i \hookrightarrow Y_{i+1}$ of closed $S$-immersions is {\it a formal deformation over an algebraic space} $S$ if there is an admissible system of fat points $\{X_i\}_{i\in I}$ over $S$ such that for all $i$ in $I$, there are flat morphisms $\varphi_i : Y_i \to X_i$ which make the following diagram commutative:
 \[\xymatrix{
&Y_0 \ar[d]^{\varphi_0} \ar@{^{(}->}[r]^{f_0} &Y_1\ar[d]^{\varphi_1}\ar@{^{(}->}[r]^{f_1} &Y_2\ar[d]^{\varphi_2} \ar@{^{(}->}[r]^{f_2}&\cdots\\
&X_0\ar@{->>}[d]^{j}\ar@{^{(}->}[r] &X_1 \ar@{^{(}->}[r]& X_2 \ar@{^{(}->}[r] &\cdots \\
&S }\]
and where $\varphi_i$ induces an isomorphism 
\begin{equation}
Y_{i-1} \cong Y_{i}\times_{X_i}X_{i-1}.
\end{equation}
\end{definition}

Given an admissible system $\{X_i\}_{i\in I}$ over $S$ and a morphism $Y \to S$ of algebraic spaces, we have the trivial formal deformation given by $Y_i := Y \times_S X_i$. More generally, given an ind-object $\mathcal{Y}$ in the category of algebraic spaces over $S$ -- i.e., the filtered colimit $\mathcal{Y} = \varinjlim_{i\in I} Y_i$ where the transition maps $Y_{i} \hookrightarrow Y_{i+1}$ are closed immersions of algebraic spaces -- and, an admissible system $\{X_i\}_{i\in I}$ with $\mathcal{X} = \varinjlim_{i\in I} X_i$, we let $\mathtt{Def}_{\mathcal{X}}$ denote the fibered category of formal deformations over $S$ with respect to the admissible system $\{X_i\}$. Thus, in particular, if $Y\to S$ is a morphism of algebraic space then the category $\mathtt{Def}_{\mathcal{X}}(Y)$ has at least one object. 

An object $\mathcal{Y}=\varprojlim_{i\in I}Y_i$ of $\mathtt{Def}_{\mathcal{X}}$ gives rise to a projective system $\{\sA_i(\mathcal{Y})\}_{i\in I}$ of algebraic spaces over $S$ defined by 
\begin{equation}
\sA_i(\mathcal{Y}) := \underline{\mbox{Hom}}_S(X_i,Y_i)^{\mathrm{red}}.
\end{equation}
We call $\sA_i(\mathcal{Y})$ {\it the truncated auto-arc space of } $\mathcal{Y}$ {\it at level }  $i$. Moreover, we may form the projective limit 
\begin{equation}
\sA(\mathcal{Y}) = \varprojlim_{i\in I}\sA_i(\mathcal{Y}),
\end{equation}
which is termed {\it the infinite auto-arc space of } $\mathcal{Y}$. If $\mathcal{Y}$ is the trivial deformation of $Y\to S$, then 
\begin{equation}
\sA_i(\mathcal{Y}) = (\underline{\mbox{Hom}}_S(X_i, Y)^{\mathrm{red}}\times_{S^{\mathrm{red}}} \underline{\mbox{End}}_S(X_i)^{\mathrm{red}})^{\mathrm{red}}
\end{equation}
In general, the group scheme $\underline{\mbox{Aut}}_S(X_i)^{\mathrm{red}}$ acts on $\sA_i(\mathcal{Y})$ by conjugation for any formal deformation $\mathcal{Y}$, and this give rise to the quotient stack 
\begin{equation}
\mathcal{M}_{i}(\mathcal{Y}) : = [\sA_i(\mathcal{Y})/\underline{\mbox{Aut}}_S(X_i)^{\mathrm{red}}].
\end{equation}
In particular, when $\mathcal{Y}$ is the trivial deformation, then 
\begin{equation}
\mathcal{M}_{i}(\mathcal{Y}) \cong \sA_i(Y)\times_{S^{\mathrm{red}}} \mathcal{M}_{JN}(X_i),
\end{equation}
where $\mathcal{M}_{JN}(X_i)$ is the moduli stack of Jordan normal forms introduced at the end of section \ref{3}. The projective limit of quotient stacks
\begin{equation}
\mathcal{M}(\mathcal{Y}) = \varprojlim_{i\in I}\mathcal{M}_{i}(\mathcal{Y})
\end{equation}
is of general interest.

\section{A generalization of the motivic Igusa zeta function}\label{gen}

\begin{definition} 
Let $S$ be an algebraic space. Let $\b{Alg}/S$ denote the category of finitely presented algebraic spaces over $S$. The {\it Grothendieck ring of algebraic spaces over } $S$, denoted by $\mbox{K}_0(\b{Alg}/S)$, is the ring formed by introducing relations on the free abelian group of isomorphism classes $\langle X \rangle$ of objects $X$ of $\b{Alg}/S$:

1. $\langle X \rangle = \langle X \setminus Y \rangle + \langle Y \rangle \mbox{ whenever } Y\hookrightarrow X \mbox{ is a locally closed } S\mbox{-immersion.}$

2. $\langle Z \rangle = \langle X \times_S \mathbb{A}_S^n\rangle \mbox{ whenever } Z \to X \mbox{ is a vector bundle of constant rank } n$.
We denote the class of an algebraic space $X$ in $K_0(\b{Alg}/S)$ by $[X]$.
The multiplicative structure is defined by $ [X]\cdot[Y] := [X\times_S Y]$. 
\end{definition}

\begin{remark}
Note that when $S= \mbox{Spec}(k)$ with $k$ a field, then $\mbox{K}_0(\b{Alg}/S)$ is isomorphic to the Grothendieck ring of algebraic varieties $\mbox{K}_0(\b{Var}/S)$. This isomorphism is induced by the natural inclusion of sets $\b{Var}/S \hookrightarrow \b{Alg}/S$. Moreover, relations of the form given in 2 above are superfluous in this case (cf., Section 1 of \cite{tor}).
\end{remark}

We define $\mathbb{L}:=[\mathbb{A}_S^1]$ and call it the {\it Leftshetz motive} over $S$. We may invert this element to obtain the {\it localized Grothendieck ring} $\mathcal{G}_S := \mbox{K}_0(\b{Alg}/S)[\mathbb{L}^{-1}]$ {\it of algebraic spaces over} $S$.  We say that an element of the power series ring $\mathcal{G}_S[[t]]$ is a {\it motivic generating series of algebraic spaces over } $S$ {\it in one variable}, or, more briefly, we will call it a motivic generating series.

\begin{definition}\label{auto}
 Let $\sX = \varinjlim_{i\in I} X_i$ be given by an admissible system of fat points $\{X\}_{i\in I}$ over $S$. Let $\mathcal{Y}=\varinjlim_{i\in I} Y_i$ be an object of $\mathtt{Def}_X$ with $Y_i$ an object of $\b{Alg}/S$ for all $i\in I$. We define the {\it auto-Igusa zeta function of } $\mathcal{Y}$ to be the motivic generating series given by 
\begin{equation}
\zeta_{\mathcal{Y}}(t) = \sum_{i\in I} [\mathcal{A}_i(\mathcal{Y})]\mathbb{L}^{-n_i}t^i
\end{equation}
where $n_i = \mbox{dim}_S(Y_i)(\mbox{rank}_S(X_i)-1)+e_i(\mbox{rank}_S(X_i)-1)$ where $e_i$ is the whole number which makes the coefficient $[\mathcal{A}_i(\mathcal{Y})]\mathbb{L}^{-n_i}$ dimensionless\footnote{It is anticipated that $e_i = \mbox{rank}_S(\Omega_{X_0/S})$ for all $i\in I$.}. 
\end{definition}

Let $\mathcal{Y}$ be as in Definition \ref{auto} and assume further that $\mathcal{M}_i(\mathcal{Y})$ is an algebraic space for all $i\in I$, we may form the motivic generating function
\begin{equation}
\eta_{\mathcal{Y}}(t) =\sum_{i\in I} [\mathcal{M}_i(\mathcal{Y})]\mathbb{L}^{-m_i}t^i 
\end{equation}
where $m_i = \mbox{dim}_S(Y_i)(\mbox{rank}_S(X_i)-1)+ e_i$ where $e_i$ is chosen to be the whole number which makes the coefficient $[\mathcal{M}_i(\mathcal{Y})]\mathbb{L}^{-m_i}$ dimensionless.  We call $\eta_{\mathcal{Y}}(t)$ the {\it canonical auto-Igusa zeta function of} $\mathcal{Y}$. 

\begin{example}
Let $\mathcal{Y}$ be the trivial formal deformation of a morphism $Y \to S$ where $S = \mbox{Spec}(A)$ with $A$ reduced with respect to the admissible system $\{X_i\}_{i\in I}$ with $X_i = \mbox{Spec}(A[t]/t^{i+1})$.  Then, 
\begin{equation}
\mathcal{A}_i(\mathcal{Y}) \cong \mathcal{L}_i(Y)\times_S \mathbb{A}_S^{i}
\end{equation}
where $\mathcal{L}_i(Y)$ denotes the classical jet space of $Y$ over $S$. 
Thus, 
\begin{equation}
\zeta_{\mathcal{Y}}(t) = \sum_{i\in I} [\mathcal{L}_i(Y)]\mathbb{L}^{-d\cdot i}t^i
\end{equation}
where $d=\mbox{dim}_S(Y)$. This is the classical motivic Igusa zeta function of $Y$ over $S$. Moreover, in this case, the canonical auto-Igusa zeta function is also the classical motivic Igusa-zeta function--i.e., in this case, we have
\begin{equation}\label{hop}
\eta_{\mathcal{Y}}(t) = \zeta_{\mathcal{Y}}(t).
\end{equation}
\end{example}

\begin{example}
Let $S=\mbox{Spec}(k)$ with $k$ an algebraically closed field of characteristic zero. Let $\sY$ be the trivial deformation of the point $\mbox{Spec}(k) \to X_0=S$, and let $X_i = \mbox{Spec}(\mathcal{O}_{C,O}/\mathfrak{m}_O^{i+1})$ where $C$ is the cuspidal cubic defined by $y^2=x^3$ over $k$ and $O$ is the origin given by the singular point $(0,0)$. 
Then,
\begin{equation}\label{cusp}
\sA_{i}(\sY) \cong \mathcal{L}_{2(i-3)}(C)\times_S\mathbb{A}_{S}^{7}
\end{equation}
 for all $i > 3$ (cf. Theorem 6.1 of \cite{STOUT2017156}). Moreover, in Section 7 of loc. cit., we use formula \ref{cusp} to explicitly calculate the auto-Igusa zeta function to obtain
 \begin{equation}
  \zeta_{\sY}(t)=\bL^{-1}+\bL t+\bL^2t^2+ \frac{(\bL^7-\bL^6)t^3+\bL^7t^4+\bL^7t^{7}}{(1-\bL t^3)(1-t)}.
\end{equation}

\end{example}
\section{Motivic Rationality for Plane Curve Singularities.}\label{last}

We will need to refer to the following Set-up below.

\begin{setup} \label{setup} Let $S= \mbox{Spec}(k)$ where $k$ is an algebraically closed field of characteristic zero. Let $C$ be an algebraic curve on a smooth surface and let $p$ be a point on $C$. We form the admissible injective system of fat points $\{X_i\}_{i\in \mathbb{N}}$ over $S$ by defining $X_i$ to be the formal neighborhood $\mbox{Spec}{(\mathcal{O}_{C,p}/\mathfrak{m}_p^{i+1})}$ of the point $p$ on $C$. Let $\mathcal{Y}=\varinjlim_{i\in\mathbb{N}} Y_i$  be any object of $\mathtt{Def}_{\mathcal{X}}$ such that $\varphi_i: Y_i \to X_i$ is smooth with $Y_i$ of pure dimension $d_i$ for all $i\in\mathbb{N}$.
\end{setup}

The main rationality result the author has obtained is the following.

\begin{theorem}
Let $\mathcal{Y}$ be a formal deformation as in Set-up \ref{setup}. Then, 
\begin{equation}\label{form}
\zeta_{\mathcal{Y}}(t) = p(t)\cdot\prod_{i=1}^{n}(1-\mathbb{L}^{a_i}t^{b_i})^{-1},
\end{equation}
where $p(t) \in \mathcal{G}_{k}[t]$ and where $a_i \in \mathbb{Z}$ and $b_i \in \mathbb{N}$. 
\end{theorem}

\begin{proof}
By assumption,  $\varphi_i : Y_i \to X_i$ is smooth, and, by shrinking $Y_i$ if necessary, $\varphi_i$  factors through an \'{e}tale morphism\footnote{The existence of an \'{e}tale morphism $\psi$ follows from the proof of Theorem 8.3.3, page 182 of \cite{stacks} whose argument is taken from the proof Theorem 8.1 of \cite{LL}. In fact, one may generalize by letting $X_i$ can be merely a Deligne-Mumford Stack.} $\psi: Y_i \to \mathbb{A}_{X_i}^{d_i}$. Therefore,
\begin{equation}
\underline{\mbox{Hom}}_S(X_i, Y_i) \cong \underline{\mbox{Hom}}_S(X_i,\mathbb{A}_{X_i}^{d_i})\times_{ \mathbb{A}_{X_i}^{d_i}} Y_i \cong \underline{\mbox{End}}_S(X_i)\times_S \mathbb{A}_S^{d_i(\mathrm{rank}_S(X_i)-1)}\times_SY_i
\end{equation}
Since $(Y_i)^{\mathrm{red}} = (Y_0)^{\mathrm{red}}$, we have
\begin{equation}
[\mathcal{A}_i(\mathcal{Y})]\mathbb{L}^{-d_i(\mathrm{rank}_S(X_i)-1)-e_i(\mathrm{rank}_S(X_i)-1)} = [Y_0][ \underline{\mbox{End}}_S(X_i)]\mathbb{L}^{-e_i(\mathrm{rank}_S(X_i)-1)}
\end{equation}
Multiplying $t^i$ on both sides and summing over $i$ in $\mathcal{G}_k[[t]]$, we have
\begin{equation}\label{coeq}
\zeta_{\mathcal{Y}}(t) = [Y_0]\sum_{i=0}^{\infty} [\underline{\mbox{End}}_S(X_i)]\mathbb{L}^{-e_i(\mathrm{rank}_S(X_i)-1)}t^i.
\end{equation}
Note that then that $\zeta_{\mathcal{Y}}(t) = [Y_0]\zeta_{C,p}(t)$, where $\zeta_{C,p}(t)$ is the auto-Igusa zeta of the algebraic germ $(C,p)$ defined in \cite{stout2}. In loc. cit., we show that this series is of the form displayed in formula \ref{form}.
\end{proof}

Part of the argument in \cite{stout2} relies on rationality results of \cite{DL1} and \cite{DL2} and their later generalizations in \cite{CL1} and \cite{CL2}.
In fact, the first part of this argument holds whenever the morphisms $\varphi_i : Y_i \to X_i$ are smooth with $\{X_i\}_{i\in I}$ any admissible injective system of fat points over any algebraic space $S$ -- i.e., regardless of whether or not $X_i$ is subscheme of a plane curve $C$, formula \ref{coeq} holds provided $\varphi_i : Y_i \to X_i$ is smooth for all $i\in I$. However, as of yet, rationality is only proven for the case when $\mathcal{Y}$ is as in Set-up \ref{setup}. Thus, there are several directions for future results: 

\begin{question}\label{q1}
 Let $\mathcal{Y}$ be a formal deformation such that $\varphi_i: Y_i \to X_i$ are smooth for all $i\in I$ with $\{X_i\}_{i\in I}$ any admissible system of fat points over an algebraic space $S$. Under what conditions on the admissible system $\{X_i\}_{i\in I}$ and on $S$ will $\zeta_{\mathcal{Y}}(t)$ be rational -- i.e., when will $\zeta_{\mathcal{Y}}(t)$ have a similar form as in formula \ref{form}?
\end{question}

\begin{question}\label{q2}
Assuming that the appropriate conditions are placed on $\{X_i\}_{i\in I}$ and $S$ through investigations of the type discussed in Question \ref{q1}, which formal deformations $\mathcal{Y}$ over $\{X_i\}_{i\in I}$ will have a zeta function $\zeta_{\mathcal{Y}}(t)$ of the form given in formula \ref{form}?
\end{question}

\begin{question}\label{q3}
Note that formula \ref{hop} gives a relationship between $\eta_{\mathcal{Y}}(t)$ and $\zeta_{\mathcal{Y}}(t)$ in the most basic case (i.e., they are both equal to the classical motivic Igusa-zeta function). In general, what is the relationship between $\eta_{\mathcal{Y}}(t)$ and $\zeta_{\mathcal{Y}}(t)$? In particular, is there a way to transfer results vis-a-vis Question \ref{q1} and Question \ref{q2} to corresponding statments about $\eta_{\mathcal{Y}}(t)$? In general, if $\zeta_{\mathcal{Y}}(t)$ is rational, must $\eta_{\mathcal{Y}}(t)$ also be rational?
\end{question}

\begin{question}\label{q4}
In the case of Set-up \ref{setup}, there are cases where explicit computations of the auto Igusa-zeta function are possible. Are similar computations possible for the canonical Igusa-zeta function $\eta_{\sY}(t)$?
\end{question}

\bibliographystyle{amsalpha}
\nocite{*}

\bibliography{RTE1}

\end{document}